\documentclass[a4paper]{amsart}

\usepackage{amsmath,amssymb,amsthm}
\usepackage{bbm}

\renewcommand{\phi}{\varphi}
\newcommand{\pphi}{\hat \phi}
\newcommand{\ppsi}{\hat \psi}

\newcommand\R{\mathbb{R}}
\newcommand\N{\mathbb{N}}

\newcommand\I{\mathbbm{1}}
\renewcommand{\P}{\mathcal P}

\newcommand{\Prel}{P^{\mathrm{{rel}}}}
\newcommand{\Drel}{D^{\mathrm{{rel}}}}

\newcommand{\eps}{\varepsilon}

\newtheorem{theorem}{Theorem}[section]

\newtheorem{corollary}[theorem]{Corollary}

\newtheorem{lemma}[theorem]{Lemma}
\newtheorem{proposition}[theorem]{Proposition}

\theoremstyle{definition}
\newtheorem{example}[theorem]{Example}
\newtheorem{remark}[theorem]{Remark}

\DeclareMathOperator{\supp}{supp}
\DeclareMathOperator{\conv}{conv}

\title{A General Duality Theorem for the Monge--Kantorovich Transport Problem}

\begin{document}
\author{Mathias Beiglb\"ock, Christian Leonard, Walter Schachermayer}
\thanks{The first author acknowledges financial support from the Austrian Science
Fund (FWF) under grant P21209. The third author acknowledges support from the Austrian Science Fund
(FWF) under grant P19456, from the Vienna Science and Technology Fund (WWTF) under grant
MA13 and by the Christian Doppler Research Association (CDG).
All authors thank A.~Pratelli for helpful discussions on the topic of this paper. We also thank R.~Balka, F.~Delbaen, M.~Elekes, M.~Goldstern and G.~Maresch for their advice. }
%\date{\today}
\date{}
\maketitle

\begin{abstract}
{\small

The duality theory of the Monge--Kantorovich transport problem is analyzed in a general
setting. The spaces $X, Y$ are assumed to be polish and equipped
with Borel probability measures $\mu$ and $\nu$. The transport
cost function $c:X\times Y \to [0,\infty]$ is assumed to be Borel.
Our main result states that in this setting there is no duality gap, provided the optimal transport problem is formulated in a suitably relaxed way.
The relaxed transport problem is defined as the limiting cost of the partial transport of masses $1-\varepsilon$ from $(X,\mu)$ to $(Y, \nu)$, as $\varepsilon >0$
tends to zero.

The classical duality theorems of H.\ Kellerer, where $c$ is lower semi-continuous or uniformly bounded, quickly follow from these general results.

%Our methods are functional analytic and rely on Fenchel's perturbation technique.
\medskip

\noindent
\emph{Keywords: Monge-Kantorovich problem, duality}
}
\end{abstract}

\section{Introduction}

We consider the \emph{Monge-Kantorovich transport problem} for
Borel probability measures  $\mu,\nu$ on polish %\footnote{As we shall only use the Borel measurability of the transport cost function $c(x,y)$, we could, equivalently, assume that $(X, \mathcal{X}, \mu)$ and $(Y, \mathcal{Y}, \nu)$ are standard Borel spaces (see \cite[Section 12]{Kech95}). In fact, the case when $X=Y=[0,1]$, equipped with Lebesgue measure on the Borel sigma-algebra, is already perfectly general as we shall make clear in the appendix in more detail.}
 spaces $X,Y$. See
\cite{Vill03,Vill09} for  an excellent  account of the theory of
optimal transportation.

The set $\Pi(\mu,\nu)$ consists of all Monge-Kantorovich
\emph{transport plans}, that is,  Borel probability measures on
$X\times Y$ which have $X$-marginal $\mu$ and $Y$-marginal $\nu$.
The \emph{transport costs} associated to a transport plan $\pi$
are given by
\begin{equation}\label{CostFunctional} \langle c,\pi\rangle =\int_{X\times Y}c(x,y)\,d\pi(x,y).\end{equation}
In most applications of the theory of optimal transport, the cost
function $c:X\times Y\to [0,\infty]$ is lower semi-continuous and
only takes values in $\R_+.$ But equation (\ref{CostFunctional})
makes perfect sense if the $[0,\infty]$-valued cost function only
is Borel measurable. We therefore assume throughout this paper
that $c:X\times Y\to [0,\infty]$ is a Borel measurable function
which may very well assume the value $+ \infty$ for ``many''
$(x,y) \in X\times Y$.

An application where the value $\infty$ occurs in a natural way is
transport between measures on Wiener space $X=(C[0,1],
\|.\|_{\infty})$, where $c(x,y)$ is the squared norm of $x-y$ in
the Cameron-Martin space, defined to be $\infty$ if $x-y$ does not
belong to this space. Hence in this situation the set
$\{y:c(x,y)<\infty\}$ has $\nu$-measure $0$, for
every $x\in X$, if the measure $\nu$ is absolutely continuous with respect to the Wiener measure on $C[0,1]$. (See \cite{FeUs02,FeUs04a,FeUs04b,FeUs06}).

\medskip

Turning back to the general problem: the (primal)
Monge-Kantorovich problem is to determine the primal value
\begin{equation}
P:=P_c:= \inf\{ \langle c, \pi\rangle:\pi\in \Pi(\mu,\nu)\} %\tag{G1}
\label{G1}
\end{equation}
and to identify a primal optimizer $\hat{\pi} \in \Pi(\mu,\nu)$.
%To formulate the dual problem, we define $\Psi(\mu,\nu)$ as the set of pairs $(\varphi,\psi)$ of integrable functions $\varphi:X\to[-\infty,\infty)$ and $\psi:Y\to[-\infty,\infty)$ which satisfy $\varphi(x)+\psi(y)\leq c(x,y)$ for all $(x,y)\in X\times Y$.
To formulate the dual problem, we define
\begin{align}\label{DualSet}
\Psi(\mu,\nu)=\left\{(\varphi,\psi):
\begin{array}{l}\varphi:X\to[-\infty,\infty),\psi:Y\to[-\infty,\infty) \mbox{ integrable,}\\
\varphi(x)+\psi(y)\leq c(x,y) \mbox{ for all }(x,y)\in X\times Y.
\end{array}\right\}.
\end{align}The dual
Monge-Kantorovich problem then consists in determining
\begin{equation}\label{SimpleJ}
D:=D_c:=\sup \left\{ \int_{X}\varphi\,d\mu+\int_{Y}\psi\,d\nu \right\}
\end{equation}
for $(\varphi,\psi)\in\Psi(\mu,\nu)$. We say that Monge-Kantorovich duality
holds true, or that {\it there is no duality gap}, if the primal value $P$ of the problem equals the dual value $D$, i.e.\ if we have
\begin{equation}%\tag{G3-A1}
\label{G3-A1}
\inf \{\langle c,\pi\rangle :\pi\in\Pi(\mu,\nu)\}=\sup \left\{\int_X \varphi\,d\mu+\int_Y \psi \,d\nu :(\varphi,\psi)\in\Psi (\mu,\nu)\right\}.
\end{equation}

There is a long line of research on these questions, initiated already by Kantorovich (\cite{Kant42}) himself and continued by numerous others %increasingly general answers to
(we mention
\cite{KaRu58,Dudl76,Dudl02,deAc82,GaRu81,Fern81,Szul82,
Mika06,MiTh06}, see also the bibliographical notes in \cite[p
86, 87]{Vill09}).

The validity of the above duality \eqref{G3-A1} was established in pleasant generality by H.~Kellerer \cite{Kell84}. He proved that there is no duality gap provided that
 $c$ is lower semi-continuous (see \cite[Theorem 2.2]{Kell84}) or just Borel
measurable and bounded by a constant\footnote{or, more generally,
by the sum $f(x)+g(y)$ of two integrable functions $f,g$.}
(\cite[Theorem 2.14]{Kell84}). In \cite{RaRu95,RaRu96} the problem
is investigated beyond the realm of polish spaces and a
characterization is given for which spaces duality holds for all
bounded measurable cost functions. We also refer to the seminal
paper \cite{GaMc96} by W.~Gangbo and R.~McCann. %W.~Gangbo and R.~McCann {\bf Mathias bitte!!} for a careful analysis of the duality theory of optimal transport.

\medskip

We now present a rather trivial example\footnote{This is essentially \cite[Example 2.5]{Kell84}.} which shows that, in general, there is a duality gap.

\begin{example}\label{ZeroOneInfty}
Consider $X=Y=[0,1]$ and $\mu=\nu$ the Lebesgue measure. Define $c$ on $X\times Y$ to be $0$ below the diagonal, $1$ on the diagonal and $\infty$ else, i.e.\
\begin{align*}
c(x,y)=\left\{
\begin{array}{cl}
0,&\mbox{ for }  0\le y<x\le 1,\\
1,&\mbox{ for }  0\le x=y\le 1,\\
\infty,&\mbox{ for }  0\le x<y\le 1.
\end{array}\right.
\end{align*}

 Then the only finite transport plan is concentrated on the diagonal and leads to costs of one so that $P=1$. On the other hand, for admissible $(\varphi,\psi)\in\Psi(\mu,\nu),$
it is straightforward to check, that $\varphi(x)+\psi(x)>0$ can hold true for at most countably many $x\in[0,1]$. Hence the dual value equals $D=0$,
so that there is a duality gap.
\end{example}

%We note a peculiar property of the primal problem of Example \ref{ZeroOneInfty}.  The transport company could decide to proceed a little bit sloppy in their task of transporting $\mu$ to $\nu$, i.e.\ they could decide to ignore a  very tiny portion $\mu_\eps$  of the mass $\mu$ and to  move only the remaining portion $(\mu-\mu_\eps)$. No matter how small the ignored mass $\mu_\eps$ is, the clever company succeeds to reduce their costs to $0$ -- the valued of the dual problem.

%\subsection{A general duality Theorem}

A common technique in the duality theory of convex optimisation is to pass to a {\it relaxed} version of the problem, i.e., to enlarge the sets over which the primal and/or
dual functionals are optimized.
We do so, for the primal problem (\ref{G1}), by requiring only the transport of a portion of mass $1-\varepsilon$ from $\mu$ to $\nu$, for every $\varepsilon >0$.
Fix $0\le\varepsilon\le 1$ and define
$$\Pi^{\eps}(\mu,\nu)=\{\pi\in\mathcal{M}_+ (X\times Y), \|\pi\|\geq 1-\varepsilon , p_X(\pi)\le \mu, p_Y(\pi)\le \nu\}.$$
Here $\mathcal{M}_+ (X\times Y)$ denotes the non-negative Borel measures $\pi$ on $X\times Y$ with norm $\|\pi\| =\pi(X\times Y)$; by $p_X (\pi)\le\mu$
(resp. $p_Y (\pi)\le\nu$) we mean that the projection of $\pi$ onto $X$ (resp. onto $Y$) is dominated by $\mu$ (resp. $\nu$). We denote by $P^\varepsilon$ the value of the
$1-\varepsilon$ partial transportation problem
\begin{align}%\tag{G7}
\label{G7} P^\varepsilon :=\inf\left\{\langle
c,\pi\rangle=\int_{X\times Y}\, c(x,y)\,d\pi(x,y):
\pi\in\Pi^{\eps}(\mu,\nu)\right\}.
\end{align}
This partial transport problem has recently been studied by L.~Caffarelli and R.~McCann \cite{CaMc06}  as well as A.~Figalli \cite{Figa09}. In their
work the emphasis is on a finer analysis of the Monge problem for the squared Euclidean distance on $\mathbb{R}^n$, and pertains to a fixed $\varepsilon>0$. In the present
paper, we do not deal with these more subtle issues of the Monge problem and always remain in the realm of the Kantorovich problem (\ref{G1}).
Our emphasis is on the limiting behavior for $\varepsilon\to 0:$ we call
\begin{align}%\tag{G8}
\label{G8}
\Prel_c:=\Prel := \lim\limits_{\varepsilon\to 0} P^\varepsilon
\end{align}
the {\it relaxed primal value} of the transport plan. Obviously this limit exists (assuming possibly the value $+ ~\infty$) and $\Prel\le P$.

\medskip

%As observed above, we find that in Example \ref{ZeroOneInfty} $\Prel=0=D$, i.e.\ there is no duality gap if we consider the relaxed version of the primal problem. Our main result asserts that this is true in full generality.

As a motivation for the subsequent theorem the reader may observe that, in Example \ref{ZeroOneInfty} above, we have $\Prel=0$ (while $P=1$). Indeed, it is possible to transport the measure $\mu \I_{[\varepsilon,1]}$ to the measure $\nu\I_{[0, 1-\varepsilon]}$ with transport cost zero by the partial transport plan $\pi =(id, id-\varepsilon)_\# ~ (\mu \I_{[\varepsilon ,1]})$.

\medskip

We now can formulate our main result.

\begin{theorem}\label{1.2}
Let $X, Y$ be polish spaces, equipped with Borel probability
measures $\mu,\nu$, and let $c:X\times Y\to [0 ,\infty]$ be Borel
measurable.

Then there is no duality gap, if the primal problem is defined in the relaxed form (\ref{G8}) while the dual problem is formulated in its usual form \eqref{SimpleJ}.
In other words, we have
\begin{align}%\tag{G9}
\label{G9} \Prel =D.
\end{align}
\end{theorem}

We observe that in (\ref{G9}) also the value $+\infty$ is possible.

\medskip

The theorem gives a positive result on the issue of duality in
the Monge--Kantorovich problem. %There remain further questions pertaining to the duality theory to be settled which are not covered by the theorem. Under which conditions do we have $P=\Prel$, and when do primal and dual optimizers $\hat{\pi}$ and $(\hat{\varphi},\hat{\psi})$ exist?
%In section 2 below we shall give a proof of Theorem \ref{1.2}. It will follow a perturbation argument, due to W.~Fenchel.
% As a by-product of the proof of Theorem \ref{1.2} we obtain in Section 2 a criterion (Proposition \ref{Prel=PIffLSC}) characterizing the validity of $P=\Prel$,
%i.e., whether the value of the optimization problem (\ref{G1}) is equal to its relaxed value.
Moreover we have $P=\Prel$ and therefore $P=D$ in any of the following cases.
\begin{enumerate}
    \item[(a)] $c$ is lower semi-continuous,
    \item[(b)] $c$ is uniformly bounded or, more generally,
    \item[(c)] $c$ is $\mu\otimes\nu$-a.s.\
finitely valued.
\end{enumerate}
Concerning (a) and (b), it is rather straight forward to check
that these assumptions imply $P=\Prel$ (see Corollaries
\ref{BoundedDuality} and \ref{LSCDuality} below). In particular,
the classical duality results of Kellerer quickly follow from
Theorem \ref{1.2}. To achieve that also property (c) is sufficient
seems to be more sophisticated and follows from \cite[Theorem
1]{BeSc08}. %, cf.\ also Theorem \ref{strong} below.

\medskip

A sufficient condition for attainment in the primal part of the Monge-Kantorovich transport problem is that the cost function $c$ is lower semi-continuous and we have nothing to add here. %Conversely  if this condition is not satisfied, there is little reason why an optimizer should exist.

To analyze the same question concerning the dual problem we need some preparation: consider the following alternative definition of $\Prel$.
One may relax the transport costs by  cutting the maximal transport  costs. I.e.\ we could alter the cost function $c$ to $c\wedge M$ for some  $M\geq0$ or to $c\wedge h$ for some $\mu\otimes \nu$-a.s.\ finite, measurable function $h:X\times Y\to [0,\infty].$  If $M$ resp.\ $h$ is large this should have a similar effect as ignoring a small mass. Indeed we will establish that
\begin{align}\label{SecondPrel}
\lim_{n\to \infty} P_{c\wedge h_n} = \Prel_c
\end{align}
for any sequence of measurable functions $h_n:X\times Y\to [0,\infty)$ increasing (uniformly) to $\infty$.

%This characterization also connects nicely to the question whether there is attainment in the dual part of the Monge-Kantorovich transport problem.
In Theorem \ref{strong} below we then prove that we have dual attainment (in the sense of \cite[Section 1.1]{BeSc08}) if and only if there exists some finite measurable function $h:X\times Y\to [0,\infty)$ so that
\begin{align}\label{RightNow} P_{c\wedge h}= \Prel_c.\end{align}

\medskip

The paper is organized as follows.

In Section \ref{DualitySection} we show Theorem \ref{1.2}. The proof is self-contained with the exception of Lemma \ref{Lnull} which is a consequence of \cite[Lemma 1.8]{Kell84}. For the convenience of the reader we provide a derivation of Lemma \ref{Lnull} in the Appendix.

Section \ref{CorollarySection} deals with  consequences of Theorem \ref{1.2}. First we re-derive the classical duality results of Kellerer. Then we establish (with the help of \cite[Theorem 2]{BeSc08}) the alternative characterization of $\Prel $ given  in \eqref{SecondPrel} and the characterization of dual attainment via \eqref{RightNow}.

\section{The Proof of the Duality Theorem}\label{DualitySection}

The proof of Theorem \ref{1.2} relies on Fenchel's perturbation
technique. We refer to the accompanying paper \cite{BeLS09b} for a
didactic presentation of this technique: there we give an
elementary version of this argument, where $X=Y=\{1,\ldots ,N\}$
equipped with the uniform measure $\mu=\nu,$ in which case the
optimal transport problem  reduces to a finite linear
programming problem.

We start with an easy result showing that the relaxed version
\eqref{G7} of the optimal transport problem is not ``too
relaxed'', in the sense that the trivial implication of the minmax
theorem still holds true.

\begin{proposition}
Under the assumptions of Theorem \ref{1.2}. we have
$$\Prel \geq D.$$
\end{proposition}

\begin{proof}
Let $(\varphi,\psi)$ be integrable Borel functions such that
\begin{align}%\tag{G15a}
\label{G15a}
\varphi(x)+\psi(y)\le c(x,y), \qquad \mbox{for every} \ (x,y) \in X\times Y.
\end{align}

Let $\pi_n \in\Pi(f_n\mu, g_n\nu)$ be an optimizing sequence for
the relaxed problem, where $f_n \le \I, g_n \le \I$, and
$\pi_n(X\times Y) =\|f_n\|_{L^1 (\mu)}= \|g_n\|_{L^1 (\nu)}$ tends to
one. By passing to a subsequence we may assume that
$(f_n)^\infty_{n=1}$ and $(g_n)^\infty_{n=1}$ converge a.s.\ to
$\I$. We may estimate
$$\liminf_{n\to\infty} \int_{X\times Y} c\,d\pi_n \geq \liminf_{n\to\infty} \left[\int_{X} \varphi f_n \,d\mu +\int_Y \psi g_n\, d\nu\right]
= \int_X\varphi\, d\mu +\int_Y \psi\, d\nu,$$
where in the last equality we have used Lebesgue's theorem on dominated convergence.
\end{proof}

The next lemma is a technical result which will be needed in the formalization of the proof of Theorem \ref{1.2}.

\begin{lemma}\label{ByHahnBanach}\label{AlmostHB}
Let  $V$ be a normed vector space, $x_0\in V$, and let $\Phi:V\to (-\infty ,\infty]$ be a positively homogeneous\footnote{By positively homogeneous we mean $\Phi (\lambda x)=
\lambda \Phi (x)$, for $\lambda\geq 0$, with the convention $0\cdot\infty =0$.}
convex function such that
    $$\liminf_{\|x-x_0\|\to 0} \Phi(x)\geq \Phi(x_0).$$
If $\Phi(x_0) < \infty$ then, for each $\eps>0$,   there exists a
continuous linear functional $v:V\to \R$ such that
\begin{align*}
  \Phi(x_0)-\eps \leq
v(x_0) \mbox{ and } \Phi(x)\geq v(x),\ \mbox{for all $x\in V$} .
\end{align*}
If $\Phi(x_0) = \infty$ then, for each $M>0$,   there exists a
continuous linear functional $v:V\to \R$ such that
\begin{align*}
  M \leq
v(x_0) \mbox{ and } \Phi(x)\geq v(x),\ \mbox{for all $x\in V$} .
\end{align*}
\end{lemma}

\begin{proof} Assume first that $\Phi(x_0) < \infty$.
Let $K=\{(x,t):x\in V, t\geq \Phi(x)\}$ be the epigraph of $\Phi$
and $\overline K$ its closure in $V\times \mathbb{R}$. Since $\Phi
$ is assumed to be lower semi continuous at $x_0$, we have
$\inf\{t:(x_0,t)\in \overline K\}= \Phi(x_0)$, hence
$(x_0,\Phi(x_0)-\eps)\notin \overline K$. By Hahn-Banach, there is
a continuous linear functional $w\in V^*\times\mathbb{R}$ given by
$w(x,t)=u(x)+s t$ (where $u\in V^*$ and $s\in \R$) and $\beta\in
\R$ such that $w(x,t) > \beta$  for  $(x,t)\in \overline K$ and
$w(x_0,\Phi(x_0)-\eps)< \beta$. By the positive homogeneity of
$\Phi$, we have $\beta<0$, hence $s>0$. Also $u(x)+s \Phi(x)\geq
\beta$ and by applying positive homogeneity once more we see that
$\beta $ can be replaced by $0$. Hence we have
\begin{align*}
u(x)+s \Phi(x)\geq 0\quad u(x_0)+s (\Phi(x_0)-\eps)<0,
\end{align*}
so just let $v(x):= -u(x)/s.$
In the case $\Phi(x_0) =\infty$ the assertion is proved analogously.
\end{proof}

We now define the function $\Phi$ to which we shall apply the previous lemma.

Let $W=L^1(\mu)\times L^1(\nu)$ and $V$ the subspace of
co-dimension one, formed by the pairs $(f,g)$ such that $\int_X
f\,d\mu =\int_Yg\,d\nu$. By $V_+=\{(f,g)\in V: \ f\geq 0,\ g\geq
0\}$ we denote the positive orthant of $V$. For
$(f,g)\in V_+,$ we define, by slight abuse of notation, $\Pi
(f,g)$ as the set of non-negative Borel measures $\pi$ on $X\times
Y$ with marginals $f\mu$ and $g\nu$ respectively. With this
notation $\Pi(\I,\I)$ is just the set $\Pi(\mu,\nu)$ introduced
above. Define $ \Phi:  V_+ \longrightarrow [0,\infty] $ by
$$ \Phi(f,g)=\inf\left\{\int_{X\times Y} c (x,y) \,d\pi(x,y):\pi\in\Pi(f,g)\right\},\quad (f,g)\in V_+,$$
which is a convex function. %: its epigraph is the epigraphic hull of the linear projection of the (convex) epigraph of $\pi\mapsto \lan{c,\pi}.$
By definition we have $\Phi(\I,\I)=P$, where $P$ is
the primal value of (\ref{G1}). Our matter of concern will be the
lower semi-continuity of the function $\Phi$ at the point
$(\I,\I)\in V_+$.

\begin{proposition}\label{Prel=PIffLSC}
Denote by $\overline \Phi:V\to[0,\infty]$ the lower
semi-continuous envelope of $\Phi$, i.e., the largest lower
semi-continuous function on $V$ dominated by $\Phi$ on $V_+$. Then
\begin{align}%\tag{G18}
\label{G18}
\overline\Phi(\I,\I)=\Prel.
\end{align}

Hence the function $\Phi$ is lower semi-continuous at $(\I,\I)$ if and only if $P=\Prel$.
\end{proposition}

\begin{proof}
Let $(\pi_n)^\infty_{n=1}$ be an optimizing sequence for the
relaxed problem \eqref{G8}, i.e., a sequence of non-negative
measures on $X\times Y$ such that
$$\lim\limits_{n\to\infty} \int_{X\times Y} c(x,y)\,d\pi_n (x,y)=\Prel,$$
$$\lim\limits_{n\to\infty} \|\pi_n\| =\lim\limits_{n\to\infty} \int_{X\times Y}1 \, d\pi_n (x,y)=1,$$
and such that $p_X(\pi_n)\le\mu$ and $p_Y (\pi_n)\le\nu$. In
particular $p_X(\pi_n)=f_n\mu$ and $p_Y(\pi_n)=g_n\mu$ with
$(f_n)^\infty_{n=1}$ (resp.\ $(g_n)^\infty_{n=1}$) converging to
$\I$ in the norm of $L^1(\mu)$ (resp.\ $L^1(\nu)$). It follows
that
$$\overline\Phi(\I,\I)\le\lim\limits_{n\to\infty} \Phi (f_n,g_n)=\Prel.$$

To prove the reverse inequality $\overline \Phi(\I,\I)\geq \Prel$,
fix $\delta>0$. We have to show that for each $\eps >0$ there is
some $\tilde \pi\in \Pi^\eps(\mu,\nu)$ such that
\begin{align}\label{Desiree}\overline \Phi(\I,\I)+\delta \geq \int c\, d\tilde\pi.\end{align} Pick $\gamma\in (0,1)$ such that $(1-\gamma)^3\geq 1-\eps$. Pick $f, g$ and $\pi\in\Pi(f,g)$ such that $\|f-\I\|_{L^1(\mu)}, \|g-\I\|_{L^1(\nu)}<\gamma$ and $\overline \Phi(\I,\I)+\delta \geq \int c\, d\pi$. We note for later use that $\|\pi\|=\|f\|_{L^1(\mu)}= \|g\|_{L^1(\nu)}\in (1-\gamma, 1+\gamma).$
%fix $1> \gamma>0 $. Assume that $\eps>0$, and $(f,g)\in V_+ $ are such that
%$\frac{1-\eps}{(1+\eps)^2}\geq \gamma $ and
%$\|f-\I\|_{L^1(\mu)},\|g-\I\|_{L^1(\nu)}\leq \eps$. Let
%take an arbitrary $\pi\in \Pi(f,g),$
Define the Borel measure $\tilde \pi\ll\pi$ on $X\times Y$ by
%\begin{align}
%\frac {d\tilde \pi }{d\pi}(x,y):= \frac1{\max (1, f(x))\cdot \max(1,g(y))},
%\end{align}
% and set $\tilde \mu:=p_X(\tilde\pi), \tilde \nu:=p_Y(\tilde \pi)$. Then $\tilde \mu\leq \mu,\tilde \nu\leq \nu $ and $$\tilde\mu(X)=\tilde\nu(Y)=\tilde\pi(X\times Y)\geq \frac{\|f\|_{L_1(\mu)}}{\|f-\I\|_{L_1(\mu)}\cdot \|g-\I\|_{L_1(\nu)}}\geq \gamma.$$
\begin{align*}
\frac {d\tilde \pi }{d\pi}(x,y):=
\frac1{(1+|f(x)-1|)(1+|g(y)-1|)},
\end{align*}
and set $\tilde \mu:=p_X(\tilde\pi), \tilde \nu:=p_Y(\tilde \pi)$.
As $\frac {d\tilde \pi }{d\pi}\le1,$ we have  $\tilde \pi\le\pi$ so that \eqref{Desiree} is satisfied. Also $\tilde \mu\leq \mu$ and $\tilde \nu\leq \nu $. Thus it remains to check that $\|\tilde\pi\|\geq 1-\eps$.

The function $F(a,b)=\frac 1{(1+a)(1+b)}$ is convex on $[0,\infty)^2$
and by Jensen's inequality we have
\begin{align}
\|\tilde\pi\|&=\|\pi\|\int F(|f(x)-1|,|g(y)-1|)\,
\tfrac{d\pi(x,y)}{\|\pi\|} \ge \\
&\geq \|\pi\|
    F\left(\tfrac{\|f-\I\|_{L^1(\mu)}}{\|\pi\|},\tfrac{\|g-\I\|_{L^1(\nu)}}{\|\pi\|}\right)
    \ge(1-\gamma)\tfrac1{(1+\gamma/(1-\gamma))^2}\geq 1-\eps,
\end{align}
as required.
%Take $\eps>0.$ If $\|\pi\|\ge 1-\eps$ and$\|f-\I\|_{L^1(\mu)},\|g-\I\|_{L^1(\nu)}\le\eps,$ this implies
%$$
%\tilde\mu(X)=\tilde\nu(Y)=\|\tilde\pi\|\geq (1-\eps)^3.
%$$
%Whence for any sequence $(f_n,g_n)^\infty_{n=1}\in V_+$ such that $\|f_n-\I\|_{L^1(\mu)} \to 0$ and $\|g_n -\I\|_{L^1(\nu)} \to 0$, we have
%$$\liminf\limits_{n\to\infty} \Phi (f_n,g_n)\geq \Prel,$$
%which readily proves (\ref{G18}).

The final assertion of the proposition is now obvious.
\end{proof}

\begin{proof}[Proof of Theorem \ref{1.2}.]
By the preceding proposition we have to show that
$$\overline\Phi(\I,\I) =D,$$
where the dual value $D$ of the optimal transport problem is
defined in \eqref{SimpleJ}.

By Lemma \ref{AlmostHB} we know that there are
sequences\footnote{The dual space $V^*$ of the subspace $V$ of
$W=L^1(\mu)\times L^1(\nu)$ equals the quotient of the dual
$L^\infty(\mu) \times L^\infty(\nu)$, modulo the annihilator
of $V$, i.e.\  the one dimensional subspace formed by the
$(\varphi,\psi)\in L^\infty (\mu)\times L^\infty(\nu)$ of the form
$(\varphi,\psi)=(a,-a)$, for $a\in\mathbb{R}$.}
$(\varphi_n,\psi_n)^\infty_{n=1} \in W^* =L^\infty(\mu)\times
L^\infty(\nu)$ such that
$$\lim\limits_{n\to\infty} \langle(\varphi_n,\psi_n),(\I,\I)\rangle= \lim\limits_{n\to\infty} \left[\int_X \varphi_n\,d\mu +\int_Y\psi_n \,d\nu\right]=\overline\Phi(\I,\I)\in[0,\infty],$$
and such that
\begin{align}%\tag{G21}
\label{G21}
\langle(\varphi_n,\psi_n),(f,g)\rangle =\langle\varphi_n, f\rangle +\langle\psi_n,g\rangle\leq \Phi(f,g),\quad \mbox{for all} \ (f,g)\in V.
\end{align}
We shall show that (\ref{G21}) implies that, for each fixed
$n\in\N$, there are representants\footnote{Strictly speaking,
$(\varphi_n,\psi_n)$ are elements of $L^\infty(\mu)\times
L^\infty(\nu)$, i.e.\ {\it equivalence classes} of functions. The
$[-\infty,\infty[$-valued Borel measurable functions
$(\tilde\varphi_n,\tilde\psi_n)$ will be properly chosen
representants of these equivalence classes.} $(\tilde\varphi_n,
\tilde\psi_n)$ of  $(\varphi_n,\psi_n)$ such that
\begin{align}%\tag{G21a}
\label{G21a}
\tilde\varphi_n(x)+\tilde\psi_n(y)\le c(x,y)
\end{align}
for {\it all} $(x,y)\in X\times Y$. Indeed, choose any
$\mathbb{R}$-valued representants $(\check\varphi_n,
\check\psi_n)$ of $(\varphi_n,\psi_n)$ and consider the set
\begin{align}%\tag{G21b}
\label{G21b}
C=\{(x,y)\in X\times Y: \check\varphi_n(x)+\check\psi_n(y)>c(x,y)\}.
\end{align}
{\it Claim:} For every $\pi\in \Pi(\mu,\nu)$ we have that $\pi(C)=0$.

Indeed, fix $\pi\in \Pi(\mu,\nu)$ and denote by $(f,g)$ the
density functions of the projections $p_X(\pi_{|C})$ and
$p_Y(\pi_{|C})$. By (\ref{G21}) we have, for $n\geq 1$,
$$\int_{X\times Y} c\I_C\, d\pi \geq \Phi(f,g)\geq\langle\varphi_n,f\rangle+\langle\psi_n,g\rangle =\int_{X\times Y} (\check\varphi_n(x) +\check\psi_n(y))\I_C \,
d\pi(x,y)$$ By the definition of $C$ the first term above can only
be greater  than or equal to  the last term if $\pi(C)=0$, which
readily shows the above claim.

Now we are in a position to apply an innocent looking, but deep
result due to H.~Kellerer \cite[Lemma 1.8]{Kell84}\footnote{For the
convenience of the reader and in order to keep the present paper
self-contained, we provide in the appendix (Lemma \ref{Lnull}) a
proof of Kellerer's lemma, which is not relying on duality
arguments.}: a Borel set $C=X\times Y$ satisfies $\pi(C)=0$, for
each $\pi\in \Pi(\mu,\nu)$, if and only if there are Borel sets
$M\subseteq X, N\subseteq Y$ with $\mu(M)=\nu(N)=0$ such that
$C\subseteq (M\times Y) \cup (X\times N)$. Choosing such sets $M$
and $N$ for the set $C$ in \eqref{G21b}, define the representants
$(\tilde\varphi_n,\tilde\psi_n)$ by $\tilde\varphi_n=\check\varphi_n
\I_{X\backslash M} -\infty\I_M$ and
$\tilde\psi_n={\check\psi_n}\I_{Y\backslash N}-\infty \I_N$. We then
have $\tilde\varphi_n(x) +\tilde\psi_n(y)\le c(x,y)$, for {\it
every} $(x,y)\in X \times Y$. As
$$
    \lim\limits_{n\to\infty}\int_X \tilde\varphi_n \,d\mu +\int_Y\tilde\psi_n\, d\nu=
    \overline\Phi(\I,\I)=\Prel,
$$
the proof of Theorem \ref{1.2} is complete.
\end{proof}

%After finishing the proof of Theorem \ref{1.2} it is time to harvest some corollaries. %we shall presently see that the classical duality results quickly follow from this general theorem.

\section{Consequences of the Duality Theorem}\label{CorollarySection}

Assume first that the Borel measurable cost function $c:X\times
Y\to [0,\infty]$ is $\mu\otimes \nu$-almost surely bounded by some constant\footnote{In fact, the same argument works
provided that $c(x,y)\leq f(x)+g(y)$ for integrable functions
$f,g$.} $M$. We then may estimate
$$P\le P^\eps +\eps M.$$

Indeed, for $\eps >0$, every partial transport plan $\pi^\eps$ with marginals $\mu^\eps \le\mu, \nu^\eps\le\nu$ and mass $\|\pi^\eps\|=1-\eps$ may be completed to a full
transport plan $\pi$ by letting, e.g.,
$$\pi=\pi^\eps+\eps^{-1}(\mu-\mu^\eps)\otimes(\nu-\nu^\eps).$$
As $c\le M$ we have $\int c\,d\pi\le\int c \,d\pi^\eps+\eps M$. This yields the following corollary due to H.~Kellerer \cite[Theorem 2.2]{Kell84}.

\begin{corollary}\label{BoundedDuality}
Let $X,Y$ be polish spaces equipped with Borel probability
measures $\mu, \nu$, and let $c:X\times Y\to [0,\infty]$ be a Borel
measurable cost function which is uniformly bounded. Then there is
no duality gap, i.e.\  $P=D$.
\end{corollary}

\medskip

To establish duality in the setup of a lower semi-continuous cost
function $c$, it suffices to note that in this setting also the
cost functional $\Phi$ is lower semi-continuous:

\begin{lemma}\label{lscOfCostFunctional}
 \cite[Lemma 4.3]{Vill09} Let $c:X\times Y\to [0,\infty] $ be lower semi-continuous and assume that a sequence of measures $\pi_n$ on $X\times Y$ converges to a transport plan $\pi\in \Pi(\mu,\nu)$ weakly, i.e.\ in the topology induced by the bounded continuous functions on $X\times Y$. Then
$$\int c\, d\pi\leq \liminf_{n\to \infty} \int c\, d\pi_n .$$
\end{lemma}

\begin{corollary}\label{LSCDuality}\cite[Theorem 2.6]{Kell84} Let $X,Y$ be polish spaces equipped with Borel probability measures $\mu, \nu$, and let $c:X\times Y\to [0,\infty ]$ be a lower semi-continuous cost function. Then there is no duality gap, i.e.\  $P=D$.
\end{corollary}

\begin{proof}
It follows from Prokhorov's theorem and Lemma
\ref{lscOfCostFunctional} that the function $\Phi:V_+\to
[0,\infty]$ is lower semi-continuous with respect to the norm
topology of $V$.
\end{proof}

We turn now to the question under which assumptions there is dual attainment.
%\subsection{The relaxed transport problem and dual attainment}

Easy examples show that one cannot expect that the  dual problem admits  \emph{integrable} maximizers
 unless the cost function satisfies certain integrability conditions with respect to $\mu$ and $\nu$ \cite[Examples 4.4, 4.5]{BeSc08}. In fact \cite[Example 4.5]{BeSc08} takes place in a very ``regular'' setting, where $c$ is squared Euclidean distance on $\R$. %This is somewhat unsatisfying, as those need not be satisfied even in a very regular setting, for instance if $c$ is the squared distance on $\R$ \cite[Example 4.4]{BeSc08} and $\nu$ equal to the measure $\mu$ shifted to the right by one.
 In this case there exist natural candidates $(\phi, \psi)$ which, however, fail to be dual maximizers in the usual sense as they are not integrable. 

The following solution was proposed in \cite[Section 1.1]{BeSc08}.
If $\varphi$ and $\psi$ are integrable functions and
$\pi\in\Pi(\mu,\nu)$ then
\begin{align}\label{G29}
\int_X \varphi \,d\mu+\int_Y \psi\, d\nu=\int_{X\times Y} ( \varphi(x)+\psi(y))\, d\pi (x,y).
\end{align}
If we drop the integrability condition on
$\varphi$ and $\psi$, the left hand side need not make sense. But if we require that $\varphi(x)+\psi(y)\le
c(x,y)$ and if $\pi$ is a finite cost transport plan, i.e.\
$\int_{X\times Y} c\,d\pi <\infty$, then the right hand side of
(\ref{G29}) still makes good sense, assuming possibly the value $-\infty$,  and we set
\begin{align*}
J_c(\phi,\psi)=\int_{X\times Y} ( \varphi(x)+\psi(y))\, d\pi (x,y).
\end{align*}
It is not difficult to show (see \cite[Lemma 1.1]{BeSc08}) that this value does not depend on the choice of the finite
cost transport plan $\pi$ and satisfies $J_c(\phi,\psi)\leq D$. Under the assumption that there exists some finite transport plan $\pi\in\Pi(\mu,\nu)$ we then say that we have \emph{dual attainment} in the optimization problem \eqref{SimpleJ} if there exist Borel measurable functions $\hat \phi:X\to [-\infty, \infty)$ and $ \hat \psi: Y\to [-\infty, \infty)$ verifying $\phi(x)+\psi(y)\leq c(x,y)$, for $(x,y)\in X\times Y$, such that
\begin{align}\label{BetterJ}
&D=J_c(\pphi,\ppsi).%,\\ \mbox{ and } &\pphi(x)+\ppsi(y)\leq c(x,y)\ \mbox{ for all } (x,y)\in X\times Y.
\end{align}

We recall a result established in \cite{BeSc08}, generalizing
Corollary \ref{BoundedDuality}. We remark that we do not know how
to directly deduce it from Theorem \ref{1.2}.
\begin{theorem}\label{VonUns}\cite[Theorems 1 and 2]{BeSc08}
Let $X,Y$ be polish spaces equipped with Borel probability
measures $\mu, \nu$, and let $c:X\times Y\to [0,\infty]$ be a Borel
measurable cost function such that $\mu\otimes
\nu(\{(x,y):c(x,y)=\infty\})=0$. Then there is no duality gap,
i.e.\  $P=D$. Moreover there exist Borel measurable functions $\phi:X\to [-\infty,\infty), \psi:Y\to [-\infty, \infty)$ so that $\phi(x)+\psi(y) \leq c(x,y)$ for all $x\in X, y\in Y$ and $J_c(\phi, \psi)=D$.
\end{theorem}
Using Theorem \ref{VonUns} we now obtain the alternative description  of $\Prel$ and the characterization of dual attainment mentioned in the introduction.
\begin{theorem}\label{strong}
Let $X, Y$ be polish spaces, equipped with Borel probability
measures $\mu,\nu$, let $c:X\times Y\to [0 ,\infty]$ be Borel
measurable and assume that there exists a finite transport plan.
For every sequence of measurable functions $ h_n:X\times Y\to [0,\infty]$, satisfying $h_n\uparrow \infty$ uniformly\footnote{By saying that $h_n$ increases to $\infty$ \emph{uniformly}, we mean that for $n$ large enough, $h_n\geq m$ for every given constant $m\in[0,\infty)$. Indeed it is crucial to insist on this strong type of convergence: one may easily construct examples where $h_n(x,y)\uparrow \infty$ for all $(x,y)\in X\times Y$ while $P_{h_n}=0$ for every $n\in \N$.} and where each $h_n$ is $\mu\otimes \nu$-a.s. finitely valued,
we have
\begin{align}\label{Prel2}
  P_{c\wedge h_n} \uparrow \Prel.
\end{align}
Moreover, the following are equivalent.
\begin{enumerate}
\item[(i)] There is dual attainment, i.e.\ there exist measurable
functions $\phi, \psi$ such that $\phi(x)+\psi(y)\leq c(x,y) $ for
$x\in X,y\in Y$ and  $\Prel=D=J_c(\phi,\psi)$. \item[(ii)] There
exists a $\mu\otimes \nu$-a.s.\ finite function $h:X\times Y \to
[0,\infty]$ such that $\Prel= P_{c\wedge h}$.
\end{enumerate}
\end{theorem}
%We note a particular consequence. If the cost function $c$ is $\mu\otimes\nu$ a.s.\ finitely valued then $$ P=P_{c}\leq \Prel \leq P,$$ hence through Theorem \ref{strong} we recover that duality in the classical sense holds, as well as that there is dual attainment.
%Theorem \ref{VonUns} is the main tool in establishing the connection between the relaxed primal problem and the dual attainment announced in the introduction.
\begin{proof}%[Theorem \ref{strong}]
Fix $(h_n)_{n\geq 0}$ as in the Statement of the Theorem.
To prove \eqref{Prel2}, note  that by  Theorem \ref{VonUns} there
exist, for each $n$, measurable functions $\phi_n:X\to
[-\infty,\infty), \psi_n:Y\to [-\infty,\infty)$ satisfying
$\phi_n(x)+\psi_n(y)\leq c(x,y)$ for all $x\in X,y\in Y$ so that
$$J_{c}(\phi_n,\psi_n)=P_{c\wedge h_n}.$$ Thus $P_{c\wedge h_n} \leq \Prel$ for
each $n$. To see that $\lim_{n\to \infty} P_{c\wedge h_n}\geq
\Prel$, fix $\eta>0$. As $D=\Prel$ there exists $(\phi,\psi)\in
\Psi(\mu,\nu)$ so that $J(\phi, \psi)>\Prel-\eta$. Note that for,
$M\geq 0$, the pair of functions  $(M\wedge (-M\vee \phi)),
M\wedge (-M\vee \psi))$ lies in $ \Psi(\mu,\nu)$. Hence we may
assume without loss of generality that $|\phi|$ and $ |\psi|$ are
uniformly bounded by some constant $M$. Pick $n$ so that
$h_n(x,y)\geq 2 M$ for all $x\in X, y\in Y$. It then follows that
$c\wedge h_n(x,y) \geq \phi(x)+\psi(y)$ for all $x\in X, y\in Y$,
hence
$$P_{c\wedge h_n}\geq J(\phi,\psi)>\Prel-\eta,$$
which shows \eqref{Prel2}.

To prove that (ii) implies  (i), apply Theorem \ref{VonUns}  to
the cost function $c\wedge h$ to obtain  functions $\phi $ and
$\psi$ satisfying $\phi(x)+ \psi(y)\leq (c\wedge h)(x,y)$ and
$J_{c\wedge h}(\phi,\psi)= P_{c\wedge h}$. Then
$J_c(\phi,\psi)=P_{c\wedge h}=\Prel=D$, hence $(\phi,\psi)$ is a
pair of dual maximizers.

To see that (i) implies (ii),  pick dual maximizers $\phi, \psi$ and set $h(x,y):=\big(\phi(x)+\psi(y)\big)_+.$
\end{proof}

We close this section with a comment concerning a possible relaxed version of the dual problem.
\begin{remark}
 Define
\begin{align} %\tag{G25a}
\label{G25a} \Drel:=\sup \left\{\int \varphi \, d\mu+ \int \psi\,
d\nu:
\begin{array}{l}
 \varphi, \psi \mbox{ integrable},\\ \varphi(x)+\psi(y)\leq c(x,y) \ \pi\mbox{-a.e.}\\
 \mbox{for every finite cost } \pi\in \Pi(\mu,\nu)
 \end{array}
 \right\}\geq D
\end{align}
where $\pi\in\Pi(\mu,\nu)$ has finite cost if $\int_{X\times Y}
c\,d\pi <\infty$. It is straightforward to verify that we still
have $\Drel\le P$. One might conjecture (and the present authors
did so for some time) that, similarly to the situation in Theorem
\ref{1.2}, duality in the form $\Drel=P$ holds without any
additional assumption. For instance this is the case in Example
\ref{ZeroOneInfty} and combining the methods of \cite{BGMS08} and
\cite{BeSc08} one may prove that $\Drel=P$ provided that the Borel
measurable cost function $c:X\times Y\to [0,\infty]$ satisfies
that the set $\{c=\infty\}$ is \emph{closed} in the product
topology of $X\times Y.$ However a rather complicated example constructed in
\cite[Section 4]{BeLS09b} shows that under the assumptions of
Theorem \ref{1.2} it may happen that $\Drel$ is strictly smaller
then $P$, i.e.\ that there still is a duality gap.
\end{remark}

\appendix
\section{Appendix}
%\section{A measurability result}

%\setcounter{section}{5}%{\value{savecounter}}
In our proof of Theorem  \ref{1.2} we made use of the following innocent looking result due H.~Kellerer:

\begin{lemma}\label{Lnull}
Let $X,Y$ be polish spaces equipped with Borel probability
measures $\mu, \nu$, let $L\subseteq X\times Y$ be a Borel set and
assume that $\pi(L)=0$ for any $\pi\in \Pi(\mu,\nu).$ Then there
exist sets $M\subseteq X, N\subseteq Y$ such that
$\mu(M)=\nu(N)=0$ and $L\subseteq M\times Y\cup X\times N$.
\end{lemma}

Lemma \ref{Lnull} seems quite intuitive and, as we shall presently see, its proof is quite natural provided that the set $L$
is \emph{compact}. However the general case is delicate and relies
on relatively involved results from measure theory. H.~Kellerer
proceeded as follows. First he established various sophisticated
duality results. Lemma \ref{Lnull} is then a consequence of the
fact that there is no duality gap in the case when the Borel measurable cost
function $c$ is uniformly bounded (Corollary \ref{BoundedDuality}). To
make the present paper more self-contained, we provide a direct proof of
Lemma \ref{Lnull} which does not rely on duality results. Still, most ideas of the subsequent proof
are, at least implicitly, contained in the work of H.~Kellerer.

\medskip

Some steps in the proof of Lemma \ref{Lnull} are (notationally)
simpler in the case when $(X,\mu)=(Y,\nu)=([0,1],\lambda)$,
therefore we bring a short argument which shows that it is
legitimate to make this additional assumption.

Indeed it is rather obvious that one may reduce to the case that
the measure spaces $X$ and $Y$ are free of atoms. A  well known
result of measure theory (see for instance \cite[Theorem
17.41]{Kech95}) asserts that for any polish space $Z$ equipped
with a continuous Borel probability measure $\sigma$, there exists
a measure preserving Borel isomorphism  between the spaces
$(Z,\sigma)$ and $([0,1],\lambda)$. Thus there exist bijections
$f:X\to [0,1], g:Y\to [0,1]$ which are measurable with measurable
inverse, such that $f_\#\mu=g_\#\nu=\lambda.$ Hence it is
sufficient to consider the case $(X,\mu)=(Y,\nu)=([0,1],\lambda)$
and we will do so from now on.

\medskip

For a measurable set $L\subseteq [0,1]^2$ we define the functional
$$m(L):=\inf\{\lambda (A)+\lambda(B): L\subseteq A\times Y\cup X\times B\}.$$
Our strategy  is to show that under the assumptions of Lemma
\ref{Lnull}, we have that $m(L)=0$. This implies Lemma \ref{Lnull}
 since we have the following result.
\begin{lemma}\label{EasyNotTrivial}
Let $L\subseteq X\times Y$ be a Borel set with $m(L) =0$. Then there exist sets $M\subseteq X, N\subseteq Y$ such that $\mu(M)=\nu(N)=0$ and $L\subseteq M\times Y\cup X\times N$.
\end{lemma}

\begin{proof}
 Fix $\eps >0$. Since $m(L)=0$, there exist sets $A_n, B_n$ such that $\mu (A_n)<1/n$ and $ \nu(B_n)< \eps 2^{-n}$ and $L\subseteq A_n\times Y\cup X\times B_n$.  Set $A:=\bigcap_{n\geq 1}  A_n, B:=\bigcup_{n\geq 1}  B_n $. Then $ \mu (A)=0, \nu(B) <\eps$ and
$$L\subseteq \bigcap_{n\geq 1} (A_n \times Y\cup X\times B)=A \times Y\cup X\times B.$$

Iterating this arguments with the roles of $X$ and $Y$ exchanged we get the desired conclusion.
\end{proof}

The next step proves Lemma \ref{Lnull} in the case where $L$ is
compact.
\begin{lemma}\label{CompactCase}
Assume that $K\subseteq [0,1]^2$ is compact  and satisfies $ \pi(K)=0$ for every $\pi\in \Pi(\lambda,\lambda).$ Then $m(K)=0$.
\end{lemma}
\begin{proof}
Assume that $\alpha:=m(K)> 0$. We have to show that there exists a
non-trivial measure $\pi$ on $X\times Y,$ i.e.\ $\pi(K)>0$ such
that $\supp \pi\subseteq K$ and the marginals of $\pi$ satisfy
$P_X(\pi)\leq \mu, P_Y(\pi)\leq \nu$. We aim to construct
increasingly good approximations $\pi_n$ of a such a measure.

 Fix $n$ large enough  and choose
$k\geq 1$ such that $\alpha/3\leq k/n \leq \alpha /2$. Since $K$
is non-empty, there exist $i_1, j_1\in \{0,\ldots, n-1\}$ such
that
$$\Big(\big(\tfrac {i_1}n, \tfrac{j_1}{n}\big)+[0, \tfrac1n]^2\Big) \cap K\neq \emptyset.$$ After $m<k$
steps, assume that we have already chosen $(i_1,j_1),\ldots,
(i_m,j_m)$. Since $2m/n<\alpha$, we have that $K$ is not covered
by
$$ \Big(\bigcup_{l=1}^m \big[\tfrac{i_l}n,
\tfrac{i_l+1}{n}\big]\Big) \times Y\cup X\times
\Big(\bigcup_{l=1}^m \big[\tfrac{j_l}n, \tfrac{j_l+1}
n\big]\Big).$$ Thus there exist $$i_{m+1}\in\{0,\ldots,
n-1\}\setminus\{i_1,\ldots, i_m\}, j_{m+1}\in\{0,\ldots,
n-1\}\setminus\{j_1,\ldots, j_m\}$$ such that $
\Big(\big(\tfrac{i_{m+1}}n, \tfrac{j_{m+1}}n\big)+[0,
\tfrac1n]^2\Big) \cap K\neq \emptyset.$ After $k$ steps we define
the measure $\pi_n$ to be the restriction of $n\cdot\lambda^2$,
(i.e.\ the Lebesgue measure on $[0,1]^2$ multplied with the
constant $n$) to the set $\bigcup_{l=1}^k \big(\frac{i_{l}}n,
\frac{j_{l}}n\big)+[0, \tfrac1n]^2$. Then the  total mass of
$\pi_n$ is bounded from below by $k/n\geq \alpha/3$ and the
marginals of $\pi_n$ satisfy $P_X(\pi_n)\leq \mu, P_Y(\pi_n)\leq
\nu$. These properties carry over to every  weak-star limit point
of the sequence $(\pi_n)$ and each such limit point $\pi$
satisfies $\supp \pi\subseteq K$ since $K$ is closed.
\end{proof}

The next lemma will enable us to reduce the case of a Borel set $L$ to the case of a compact set $L$.

\begin{lemma}\label{CapacityLemma}
Suppose that a Borel set  $L\subseteq [0,1]^2$ satisfies $m(L)>0$. Then there exists a compact set $K\subseteq L$ such that $m(K)>0$ .
\end{lemma}

Lemma \ref{CapacityLemma} will be deduced from \emph{Choquet's
capacitability Theorem}.\footnote{It seems worth noting that
Kellerer also employs the Choquet capacitability Theorem.} Before
we formulate this result we introduce some notation. Given a
compact metric space $Z$, a capacity on $Z$ is a map
$\gamma:\mathcal {P}(Z)\to \R_+$ such that:
\begin{enumerate}
\item $A\subseteq B \Rightarrow \gamma(A)\leq \gamma(B)$.
\item  $A_1\subseteq A_2\subseteq \ldots \Rightarrow  \sup_{n\geq 1} \gamma(A_n)=\gamma (\bigcup_{n\geq 1} A_n).$
%\item For every compact set $K$, $\gamma(K)< \infty$  and for  every number $r>0$ such that $\gamma (K)<r$ there is an open set $O\supseteq  K$ such that $\gamma (O)<r$.
%\end{enumerate}
%If the space $Z$ is compact metric, condition 3. can be replaced by
%\begin{enumerate}
\item For  every sequence $K_1\supseteq K_2 \supseteq \ldots$ of
compact sets we have  $\inf_{n\geq 1} \gamma(K_n)=\gamma
(\bigcap_{n\geq 1} K_n).$
\end{enumerate}
The typical example of a capacity is the \emph{outer measure } associated to a finite Borel measure.

\begin{theorem}[Choquet capacitability Theorem]\label{CCT} See \cite{Choq59} and also \cite[Theorem 30.13]{Kech95}.\
Assume that $\gamma$ is a capacity on a polish space $Z$. Then
$$\gamma (A)=\sup\{ \gamma(K):K\subseteq A, \mbox{ $K$
compact}\}$$ for every Borel\footnote{In fact, the assertion of
the Choquet capacitability Theorem is true for the strictly larger
class of \emph{analytic sets}.} set $A\subseteq Z$.
\end{theorem}

\begin{proof}[Proof of Lemma \ref{CapacityLemma}]
We cannot apply Theorem \ref{CCT} directly to the functional $m$
since $m$ fails to be a capacity, even if it is extended in a
proper way to all subsets of $[0,1]^2.$ A clever trick\footnote{We
thank Rich\'ard Balka and M\'arton Elekes for showing us this
argument (private communication).} is to replace $m$ by the
mapping $\gamma :\P([0,1]^2)\to [0, 2]$, defined by $$\gamma
(L):=\inf \Big\{\textstyle{\int} f\, d\lambda: f:[0,1]\to [0,1],
f(x)+f(y)\geq \I_L(x,y) \mbox { for } (x,y)\in [0,1]\Big\}.$$ We
then have:
\begin{enumerate}
\item[a.] For any Borel set $A\subseteq [0,1]^2$ we have $\gamma(L)\leq m(L)\leq 4 \gamma(L).$
\item[b.] $\gamma$ is a capacity.
\end{enumerate}

To see that (a) holds true notice that $f(x)+f(y)\geq \I_L(x,y)$
implies $L\subseteq \{f\geq 1/2\} \times Y\cup X\times \{f\geq
1/2\}$ and that $L\subseteq A\times Y\cup X\times B$ yields
$\I_{A\cup B}(x)+\I_{A\cup B}(y)\geq \I_L(x,y).$

To prove (b)  it remains to check that $\gamma$ satisfies
properties (2) and (3) of the capacity definition. To see
continuity from below, consider a sequence of sets $A_1\subseteq
A_2 \subseteq \ldots$ increasing to $A$. Pick a sequence of
functions $f_n$ such that $ f_n(x)+f_n(y)\geq \I_{A_n}(x,y)$
point-wise and $\int f\, d\lambda < \gamma(A_n)+1/n$ for each
$n\geq 1$. By Komlos' Lemma there exist functions $g_n\in \conv
\{f_n, f_{n+1},\ldots\}$ such that the sequence $(g_n)$ converges
$\lambda$-a.s.\ to a function $g:[0,1]\to [0,1]$. After changing
$g$ on a $\lambda$-null set if necessary, we have that
$g(x)+g(y)\geq \I_{A}(x,y) $ point-wise. By dominated convergence,
$\int g\, d\lambda= \lim_{n\to \infty} \int g_n\, d\lambda \leq
\lim_{n\to \infty} \gamma(A_n)+1/n = \gamma(A)$. Thus $\gamma $
satisfies property 2. The proof of (3) follows precisely the same
scheme.
\medskip

An application of Choquet's Theorem \ref{CCT} now finishes the
proof of Lemma \ref{CapacityLemma}.
\end{proof}
We have done all the preparations to prove Lemma \ref{Lnull} and now summarize the
necessary steps.
\begin{proof}[Proof of Lemma \ref{Lnull}]
As discussed above, we may assume w.l.g.~that $(X,\mu)=(Y,\nu)=([0,1],\lambda).$ Suppose that the Borel set $L\subseteq [0,1]^2$ satisfies $\pi(L)=0$ for all $\pi\in \Pi(\mu,\nu)$.
Striving for a contradiction, we assume that $m(L)>0$. By Lemma \ref{CapacityLemma}, we find that there exists a compact set $K\subseteq L$ such that $m(K)>0$. By Lemma \ref{CompactCase}, there is a measure $\pi\in \Pi(\mu,\nu)$ such that $ \pi(K)>0$, hence also $\pi(L)>0$ in contradiction to our assumption. Thus $m(L)=0$. By Lemma \ref{EasyNotTrivial} we may conclude that there exist sets $M\subseteq X, N\subseteq Y, \mu(X)=\nu(N)=0 $ such that $L\subseteq M\times Y \cup X\times N$ hence we are done.
\end{proof}

\def\ocirc#1{\ifmmode\setbox0=\hbox{$#1$}\dimen0=\ht0 \advance\dimen0
  by1pt\rlap{\hbox to\wd0{\hss\raise\dimen0
  \hbox{\hskip.2em$\scriptscriptstyle\circ$}\hss}}#1\else {\accent"17 #1}\fi}

\end{document}